\documentclass[11pt,a4paper]{amsart}
\usepackage[utf8]{inputenc}
\usepackage{amsmath}
\usepackage{amsfonts}
\usepackage{amssymb}
\usepackage{amsthm}
\usepackage{setspace}
\usepackage{hyperref}
\usepackage{graphicx}
\usepackage{stmaryrd}
\usepackage[shortlabels]{enumitem}
\usepackage{mathtools}
\usepackage{subcaption}
\usepackage[T1]{fontenc}
\usepackage{xcolor}
\usepackage{float}
\usepackage{tikz-cd}
\usepackage{tikz}
\usepackage[noadjust]{cite}
\usepackage{cite}
\usepackage{ragged2e}
\usepackage[margin=1.2in]{geometry}
\graphicspath{ {./images/} }
\newtheorem{theorem}{Theorem}[section]

\newtheorem{corollary}[theorem]{Corollary}

\theoremstyle{definition}
\newtheorem{definition}[theorem]{Definition}

\newcommand{\Aut}{\mathrm{Aut}}
\newcommand{\Out}{\mathrm{Out}}

\newcommand{\Fix}{\mathrm{Fix}}


\title{On the twisted conjugacy problem for large-type Artin groups}
\author{Mart\'in Blufstein}
\address[M.~Blufstein]{Department of Mathematical Sciences, University of Copenhagen, 2100 Copenhagen, Denmark}
\email{mblufstein@dm.uba.ar}
\author{Motiejus Valiunas}
\address[M.~Valiunas]{Instytut Matematyczny, Uniwersytet Wroc{\l}awski, plac Grunwaldzki 2, 50-384 Wroc{\l}aw, Poland}
\email{motiejus.valiunas@math.uni.wroc.pl}
\date{}

\begin{document}

\begin{abstract}
\centering \justifying We show that the twisted conjugacy problem is solvable for large-type Artin groups whose outer automorphism group is finite, generated by graph automorphisms and the global inversion.
This includes XXXL Artin groups whose defining graph is connected, twistless, and not an even edge; and large-type Artin groups whose defining graph admits a twistless hierarchy terminating in twistless stars. 
\end{abstract}

\maketitle 

\section{Introduction}

Let $G$ be a group and $\varphi \in \Aut(G)$.
The group $G$ has solvable \emph{$\varphi$-twisted conjugacy problem} ($TCP_\varphi(G))$ if there is an algorithm that given $u,v\in G$ can decide if there exists a $z\in G$ such that $u = \varphi(z) v z^{-1}$.
It has solvable \emph{twisted conjugacy problem} ($TCP(G)$) if there is an algorithm that given $\varphi \in \Aut(G)$ and $u,v \in G$ can decide if there exists a $z\in G$ such that $u = \varphi(z) v z^{-1}$.
When $\varphi$ is trivial, $TCP_{id}(G)$ is known as the \emph{conjugacy problem} ($CP(G)$).

The twisted conjugacy problem has proven to be much harder to solve than the conjugacy problem, and few positive results are known.
In the context of Artin groups, it is only known for braid groups \cite{gonzalezmeneses2014twisted}, and for Artin groups whose defining graph has two vertices \cite{crowe2024twisted1, crowe2024twisted2}.
In this article we show that some large-type Artin groups have solvable twisted conjugacy problem.
Our result is the first one to hold for generic Artin groups in a probabilistic sense (see \cite[Theorem~1.2]{goldsborough2023random}).
Here is the precise statement, for definitions see Section \ref{sec:Artin}.

\begin{theorem}\label{thm:TCP}
    Let $A_\Gamma$ be a large-type Artin group such that $\Out(A_\Gamma)$ is generated by the graph automorphisms and the global inversion.
    Then $TCP(A_\Gamma)$ is solvable.
\end{theorem}

As a corollary, we list the cases where the condition in Theorem~\ref{thm:TCP} is known to be true.
The case where $\Gamma$ is a single edge with an even label follows from \cite[Theorem~3.28]{crowe2024twisted2}, and all the remaining cases follow from Theorems~\ref{thm:TCP}~and~\ref{thm:Out-finite}.

\begin{corollary}\label{cor:TCP}
    Let $A_\Gamma$ be a large-type Artin group. Then $TCP(A_\Gamma)$ is solvable if one of the following holds: 
    \begin{itemize}
        \item $A_\Gamma$ is XL, satisfies the COST property and $\Gamma$ is connected and twistless; or
        \item $\Gamma$ is connected without valence one vertices, and $A_\Gamma$ is fundamental. 
    \end{itemize}
\end{corollary}

Here twistless means that the graph $\Gamma$ does not have a cut-vertex or a separating edge, and the COST and fundamental properties are properties ensuring a good behaviour for patterns of intersections of some fixed point sets in the Deligne complex.
For precise definitions see Section \ref{sec:Artin}.

In \cite[Proposition 6.1]{bmv2024homomorphisms} is was shown that XXXL Artin groups satisfy the COST property, and in \cite[Proposition 5.11]{huang2024rigidity} it was shown that if $\Gamma$ has a twistless hierarchy terminating on twistless stars, then $A_\Gamma$ is fundamental.
These are the two biggest explicit families known to satisfy one of the conditions in Corollary~\ref{cor:TCP}, but they are conjectured to hold in more generality.

Twisted conjugacy of Artin groups from these two families have recently been studied in terms of the number of twisted conjugacy classes: in \cite[Theorem~4]{SorokoVaskou} it was shown that these groups $A_\Gamma$ satisfy the property $R_\infty$, i.e.\ for any $\varphi \in \Aut(A_\Gamma)$, the number of $\varphi$-twisted conjugacy classes is infinite.
Similar properties, both in terms of the number of $\varphi$-twisted conjugacy classes and the solubility of the $\varphi$-twisted conjugacy problem, have previously been studied for specific automorphisms $\varphi$ of XL Artin groups \cite{Juhasz2010}.

As a corollary of Theorem \ref{thm:TCP} we get the following (see \cite[\S 2]{soton69565} for the definition of algorithmic and Section~\ref{sec:Artin} for the statement of Theorem~\ref{thm:s.e.s.}).

\begin{corollary}\label{cor:application}
    Let $A_\Gamma$ be an Artin group as in Theorem \ref{thm:TCP}.
    Then, for any algorithmic short exact sequence of groups
    \[ 1 \to A_\Gamma \to G \to H \to 1 \]
    with $H$ satisfying the hypotheses of Theorem \ref{thm:s.e.s.}, the group $G$ has solvable conjugacy problem.
\end{corollary}

In particular, if an Artin group $A_\Gamma$ satisfies the assumptions of Theorem~\ref{thm:TCP}, then any group $G$ containing $A_\Gamma$ as a finite-index normal subgroup has solvable conjugacy problem. In fact, we can deduce from our proofs that such groups $G$ are systolic and therefore biautomatic \cite[Theorem~E]{JanuszkiewiczSwiatkowski2006}.

\subsection*{Acknowledgements}
The first author wants to thank Enric Ventura for pointing out Theorem \ref{thm:s.e.s.}. The second author was partially supported by the National Science Centre (Poland) grant 2022/47/D/ST1/00779.

\section{Preliminaries} \label{sec:Artin}

A \emph{presentation graph} is a finite simplicial graph $\Gamma$ where every edge connecting vertices $s$ and $t$ is labelled by an integer $m_{st} \geq 2$.
Given a presentation graph $\Gamma$, the associated \emph{Artin group} $A_\Gamma$ is the group given by the following presentation:
\[A_\Gamma \coloneqq \langle s \in V(\Gamma) \mid \Pi(s, t; m_{st}) = \Pi(t, s; m_{st})  ~ \mbox{ whenever } s,t \mbox{ are adjacent in }\Gamma \rangle, \]
where $\Pi(x, y; k)$ denotes the alternating product of $xyxy\cdots$ with $k$ letters.
Similarly the associated \emph{Coxeter group} $W_\Gamma$
is the group given by the following presentation:
\[W_\Gamma \coloneqq \langle s \in V(\Gamma) \mid s^2=1 \mbox{ for all } s\in V(\Gamma), (st)^{ m_{st}}=1 ~ \mbox{ whenever } s,t \mbox{ are adjacent in }\Gamma \rangle, \]

Given a subset $S\subseteq V(\Gamma)$, the subgroup $\langle S \rangle \leqslant A_\Gamma$ is called a \emph{standard parabolic subgroup} of $A_\Gamma$.
By a result of van der Lek \cite[Theorem~II.4.13]{van1983homotopy}, the subgroup $\langle S \rangle \leqslant A_\Gamma$ is isomorphic to the Artin group whose defining presentation graph is the subgraph of $\Gamma$ induced by $S$.

An Artin group $A_\Gamma$ is:
\begin{itemize}
    \item \emph{large-type} if all labels in $\Gamma$ are at least $3$;
    \item \emph{XL} if all labels in $\Gamma$ are at least $4$;
    \item \emph{XXXL} if all labels in $\Gamma$ are at least $6$;
    \item \emph{free-of-infinity} if $\Gamma$ is a complete graph;
    \item \emph{hyperbolic-type} if $W_\Gamma$ is a hyperbolic group;
    \item \emph{spherical-type} if $W_\Gamma$ is finite.
\end{itemize}

\begin{definition}
    The \emph{Deligne complex} of an Artin group $A_\Gamma$ is the simplicial complex $D_\Gamma$ defined as follows: \begin{itemize}
    \item Vertices correspond to left cosets of standard parabolic subgroups of spherical type.
    \item For every $g \in A_\Gamma$ and for every chain of induced subgraphs $\Gamma_0 \subsetneq \cdots \subsetneq \Gamma_k$ with $A_{\Gamma_0}, \ldots, A_{\Gamma_k}$ of spherical type, there is a $k$-simplex with vertices $gA_{\Gamma_0}, \ldots, gA_{\Gamma_k}$.
    \end{itemize}
    Equivalently, the Deligne complex $D_\Gamma$ is the geometric realisation of the poset of left cosets of standard parabolic subgroups of spherical type.
    
    The group $A_\Gamma$ acts on $D_\Gamma$ by left multiplication on left cosets, and we denote by $K_{\Gamma}$ the subcomplex induced by the vertices of the form $1 \cdot A_{\Gamma'}$.
    A \emph{standard tree} of $D_\Gamma$ is the fixed-point set of a conjugate of a standard generator of $A_\Gamma$ (see \cite{martin2022acylindrical}).  
\end{definition}

In \cite{bmv2024homomorphisms} the first author, Martin and Vaskou studied homomorphisms between large-type Artin groups, and to do so introduced the following notions.

\begin{definition}\label{def:cycle_standard_trees_property}
    A \emph{cycle of standard trees} is a sequence $T_1, \ldots, T_n$ of distinct standard trees of the Deligne complex $D_\Gamma$ such that for every $i \in \mathbb{Z}/n\mathbb{Z}$, the intersection $T_i \cap T_{i+1}$ is a vertex, and such that: 
    \begin{itemize}
        \item for every $i \in \mathbb{Z}/n\mathbb{Z}$, the generators $x_i$ and $x_{i+1}$ of $\Fix(T_i)$ and $\Fix(T_{i+1})$ respectively generate a dihedral Artin group, 
        \item for every distinct $i, j$ with $j \neq i\pm1$, the generators of $\Fix(T_i)$ and $\Fix(T_j)$ generate a non-abelian free group.
    \end{itemize}
    
    An Artin group $A_\Gamma$ satisfies the \emph{Cycle of Standard Trees Property} (COST) if the following holds for any cycle of standard trees: let $T_1, \ldots, T_n$ be a cycle of standard trees in $D_\Gamma$, and let $\gamma$ be the loop of $D_\Gamma$ obtained by concatenating, for $i \in \mathbb{Z}/n\mathbb{Z}$, the geodesic segments $\gamma_i$ of $T_i$ between the vertices $T_i\cap T_{i-1}$ and $T_i\cap T_{i+1}$.
    Then there exists an element $g\in A_\Gamma$ such that $\gamma$ is contained in $gK_{\Gamma}$.
\end{definition}

In a similar vein, Huang, Osajda and Vaskou introduced the following notions to study rigidity of large-type Artin groups.

\begin{definition}
    Let $\Gamma$ be a connected presentation graph with $V(\Gamma) = \{s_1,\ldots,s_n\}$.
    Then $A_\Gamma$ is \emph{fundamental} if for every distinct standard trees $T_1,\ldots,T_n$ in $D_\Gamma$ such that
    \begin{itemize}
        \item if $s_i$ and $s_j$ are adjacent in $\Gamma$ then $T_i \cap T_j$ is a single vertex $v_{ij}$, and all the $v_{ij}$'s are distinct; and
        \item if $s_i$ and $s_j$ are not adjacent, then $T_i \cap T_j = \emptyset$;
    \end{itemize}
    there is a unique element $g\in A_\Gamma$ such that the translate $gK_\Gamma$ contains all the vertices of the form $v_{ij}$ and contains a vertex of the form $g\langle s_{t_i} \rangle \in T_i$ for every $i\in\{1,\ldots,n\}$.
\end{definition}

\begin{definition}
    We say that an edge $e$ in a connected graph $\Gamma$ is \emph{separating} if there exist two induced connected subgraphs $\Gamma_1$, $\Gamma_2$ such that $\Gamma = \Gamma_1 \cup \Gamma_2$ and $\Gamma_1 \cap \Gamma_2 = e$.
    A graph $\Gamma$ is said to be \emph{twistless} if it does not have a cut-vertex or a separating edge.
    An \emph{admissible decomposition} of $\Gamma$ is a pair of full subgraphs $\Gamma_1$ and $\Gamma_2$ of $\Gamma$ such that $\Gamma = \Gamma_1 \cup \Gamma_2$.
    The decomposition is \emph{twistless} if $\Gamma_1 \cap \Gamma_2$ is not the empty set, a vertex, or an edge.
    The graph $\Gamma$ has a \emph{hierarchy} terminating in a class of graphs $\mathcal{C}$, if it is possible to start with $\Gamma$ and perform admissible decompositions a finite number of times to obtain a collection of graphs in $\mathcal{C}$.
    The hierarchy is \emph{twistless} the decompositions at each step are twistless.
    The graph $\Gamma$ is a \emph{twistless star} if it has a central vertex and is twistless.
\end{definition}

In \cite[Proposition~6.1]{bmv2024homomorphisms} it was shown that XXXL Artin groups satisfy the COST property, and in \cite[Proposition 5.11]{huang2024rigidity} it was shown that if $\Gamma$ has a twistless hierarchy terminating on twistless stars, then $A_\Gamma$ is fundamental.
Both of these properties are conjectured to hold for bigger families.
Combining results from \cite{bmv2024homomorphisms} and \cite{huang2024rigidity}, we can obtain the following description of the automorphism group of some classes of large-type Artin groups.
Here graph automorphisms are automorphisms of $A_\Gamma$ induced by automorphisms of the labelled graph $\Gamma$, and the global inversion is the automorphism that sends every generator to its inverse.

\begin{theorem}[{\cite[Corollary~1.7]{bmv2024homomorphisms},\cite[Theorem~6.6]{huang2024rigidity}}]
\label{thm:Out-finite}
    Let $A_\Gamma$ be a large-type Artin group such that one of the following holds:
    \begin{itemize}
        \item $A_\Gamma$ is XL, satisfies the COST property and $\Gamma$ is connected, twistless, and not an even edge; or
        \item $\Gamma$ is connected without valence one vertices, and $A_\Gamma$ is fundamental. 
    \end{itemize}
    Then $\Aut(A_\Gamma)$ is generated by the conjugations, the graph automorphisms, and the global inversion.
    In particular, $\Out(A_\Gamma)$ is finite.
\end{theorem}

Large-type Artin groups also enjoy the property of being systolic, and as a consequence they are biautomatic \cite[Theorem 13.1]{JanuszkiewiczSwiatkowski2006} and have solvable conjugacy problem \cite[Theorem~2.5.7]{Epsteinetal1992}.
Systolicity was shown by Huang and Osajda in \cite{huang2020large}, where they constructed a thickening of the Cayley complex of a large-type Artin group and proved that it is systolic.

If $G$ is a group, $F$ is a normal subgroup of $G$ and $g\in G$, we denote by $\varphi_g$ the automorphism of $F$ induced by conjugation by $g$.
In \cite{soton69565} Bogopolski, Martino and Ventura proved the following theorem (see \cite[\S 2]{soton69565} for the definition of algorithmic).
Here, if $A$ is a subgroup of $\Aut(F)$ for a group $F$, we say that $A$ is \emph{orbit decidable} if there is an algorithm which, for any two $u,v\in F$ decides whether there exists $\varphi \in A$ such that $\varphi(u)$ and $v$ are conjugate in $G$.

\begin{theorem}[{\cite[Theorem~3.1]{soton69565}}]\label{thm:s.e.s.}
Let
\[1 \to F \xrightarrow{\alpha} G \xrightarrow{\beta} H \to 1\]
be an algorithmic short exact sequence of groups such that
\begin{enumerate}
    \item $CP(H)$ is solvable, and
    \item for every $1 \neq h \in H$, the subgroup $\langle h \rangle$ has finite index in its centralizer $C_H(h)$, and there is an algorithm which computes a finite set of coset representatives, $z_{h,1},\ldots,z_{h,t_h} \in H$,
    \[C_H(h) = \langle h \rangle z_{h,1} \sqcup \cdots \sqcup \langle h \rangle z_{h,t_h}.\]
\end{enumerate}
Moreover, suppose that $TCP(F)$ is solvable. Then, the following are equivalent:
\begin{enumerate}[(a)]
    \item $CP(G)$ is solvable,
    \item the action subgroup $A_G = \{\varphi_g \mid g \in G\} \leqslant \Aut(F)$ is orbit decidable.
\end{enumerate}
\end{theorem}

\section{Proofs}

\begin{proof}[Proof of Theorem \ref{thm:TCP}]
Let $\varphi \in \Aut(A_\Gamma)$, and consider the group $A_\Gamma \rtimes \langle \varphi \rangle$, where $\langle \varphi \rangle$ is a cyclic group of order $k$ if $\varphi$ has order $k < \infty$ and an infinite cyclic group otherwise.
By \cite[Proposition~4.1]{soton69565}, solvability of $CP(A_\Gamma \rtimes \langle \varphi \rangle)$ implies solvability of $TCP_\varphi(A_\Gamma)$.
So in order to prove Theorem~\ref{thm:TCP}, it suffices to show that $A_\Gamma \rtimes \langle \varphi \rangle$ is systolic.
We will prove that the action of $A_\Gamma$ on the thickening of the Cayley complex extends to an action of $A_\Gamma \rtimes \langle \varphi \rangle$.
If $\varphi$ has finite order, then $A_\Gamma \rtimes \langle \varphi \rangle$ is a finite extension of $A_\Gamma$ and therefore the latter action is geometric.

We proceed to describe the thickening of the Cayley complex constructed in \cite[Definitions 3.7 and 5.3]{huang2020large}.
Let $\tilde{K}_\Gamma$ be the Cayley complex of $A_\Gamma$ associated to the standard presentation.
For each edge in $\Gamma$ between vertices $s,t \in V(\Gamma)$ there is a relation of length $2m_{st}$ that lifts to copies in $\tilde{K}_\Gamma$.
Subdivide each of these lifts to a $2m_{st}$-gon by adding $m_{st}-2$ interior vertices, and call each of these subdivided $2$-cells a precell (see Figure~\ref{fig:precell}).
Now, if two precells $C_1$, $C_2$ intersect at more than one edge, first connect interior vertices $v_1 \in C_1$, $v_2 \in C_2$ such that both $\{v_1\} \cup e$ and $\{v_2\} \cup e$ span triangles for some edge $e \subset C_1 \cap C_2$, and then add edges between interior vertices of $C_1$ and $C_2$ to form a zigzag as in Figure~\ref{fig:zigzag}.
The flag completion of this complex is the desired thickening.

It is clear that if $\varphi\in\Aut(A_\Gamma)$ is a graph automorphism or the global inversion, then it induces an automorphism of the thickening of $\tilde{K}_\Gamma$, and since $\varphi$ has finite order we get that $A_\Gamma \rtimes \langle \varphi \rangle$ acts geometrically.
Hence $A_\Gamma \rtimes \langle \varphi \rangle$ is systolic and has solvable conjugacy problem, so $TCP_\varphi(A_\Gamma)$ is solvable.
Since the classes of such automorphisms $\varphi$ generate $\Out(A_\Gamma)$ by Theorem~\ref{thm:Out-finite}, it follows that $TCP(A_\Gamma)$ is solvable as well, as required.
\end{proof}

\begin{figure}[ht]
    \centering
    \begin{tikzpicture}[scale=0.8]
        \draw[very thick] (0,0) -- (1,2) -- (11,2) -- (12,0) -- (11,-2) -- (1,-2) -- cycle;
        \draw (0,0) -- (12,0);
        \foreach \x in {0,...,4} {
            \draw ({2*\x+1},-2) -- ({2*\x+3},2);
            \draw ({2*\x+1},2) -- ({2*\x+3},-2);
            \draw [fill=white] ({2*\x+2},0) circle (2pt);
            \draw [very thick,->] ({2*\x+1.9},{2*(-1)^\x}) -- ({2*\x+2.1},{2*(-1)^\x}) node [yshift={8pt*(-1)^\x}] {$t$};
            \draw [very thick,->] ({2*\x+1.9},{-2*(-1)^\x}) -- ({2*\x+2.1},{-2*(-1)^\x}) node [yshift={-7pt*(-1)^\x}] {$s$};
        }
        \draw [very thick,->] (0.45,0.9) -- (0.55,1.1) node [left,yshift=3pt] {$s$};
        \draw [very thick,->] (0.45,-0.9) -- (0.55,-1.1) node [left,yshift=-3pt] {$t$};
        \draw [very thick,->] (11.45,1.1) -- (11.55,0.9) node [right] {$s$};
        \draw [very thick,->] (11.45,-1.1) -- (11.55,-0.9) node [right] {$t$};
    \end{tikzpicture}
    \caption{Subdivision of a precell with $m_{st}=7$.}
    \label{fig:precell}
\end{figure}
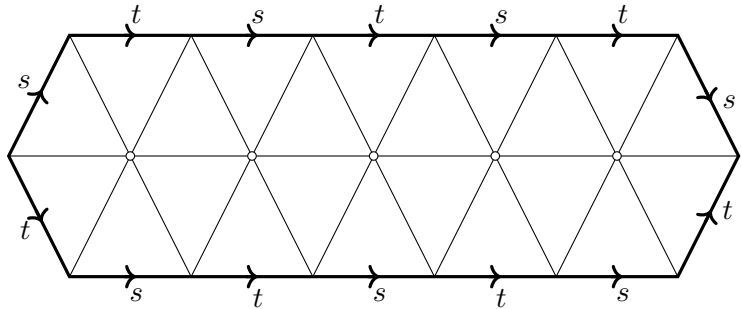

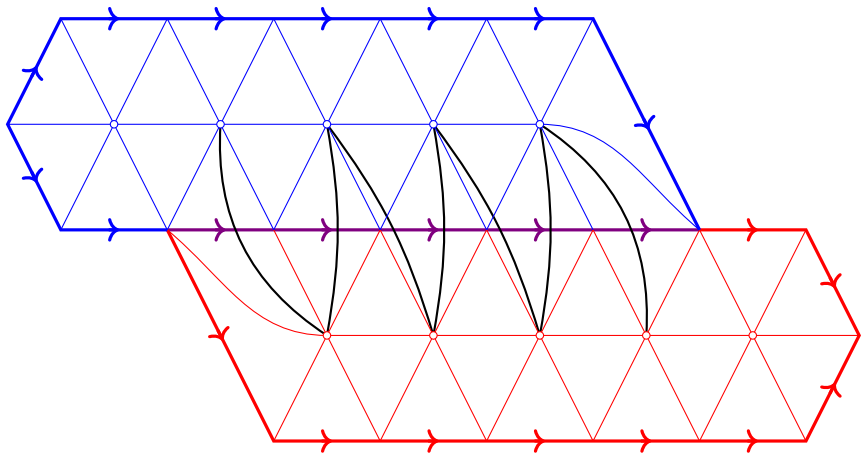
\begin{figure}[ht]
    \centering
    \begin{tikzpicture}[scale=0.7]
        \draw [very thick,red!50!blue] (-1,2) -- (9,2);
        \foreach \x in {0,...,4} \draw [very thick,red!50!blue,->] ({2*\x-0.1},2) -- ({2*\x+0.1},2);
        \draw [thick] (0,4) to[bend right] (2,0) to[bend right=10] (2,4) to[bend left=10] (4,0) to[bend right=10] (4,4) to[bend left=10] (6,0) to[bend right=10] (6,4) to[bend left] (8,0);
        \begin{scope}[red]
        \draw[very thick] (9,2) -- (11,2) -- (12,0) -- (11,-2) -- (1,-2) -- (-1,2);
        \draw (-1,2) to[out=-40,in=180] (2,0) -- (12,0);
        \foreach \x in {0,...,4} {
            \draw ({2*\x+1},-2) -- ({2*\x+3},2);
            \draw ({2*\x+1},2) -- ({2*\x+3},-2);
            \draw [fill=white] ({2*\x+2},0) circle (2pt);
            \draw [very thick,->] ({2*\x+1.9},-2) -- ({2*\x+2.1},-2);
        }
        \draw [very thick,->] (9.9,2) -- (10.1,2);
        \draw [very thick,->] (-0.05,0.1) -- (0.05,-0.1);
        \draw [very thick,->] (11.45,1.1) -- (11.55,0.9);
        \draw [very thick,->] (11.45,-1.1) -- (11.55,-0.9);
        \end{scope}
        \begin{scope}[blue,rotate=180,yshift=-4cm,xshift=-8cm]
        \draw[very thick] (9,2) -- (11,2) -- (12,0) -- (11,-2) -- (1,-2) -- (-1,2);
        \draw (-1,2) to[out=-40,in=180] (2,0) -- (12,0);
        \foreach \x in {0,...,4} {
            \draw ({2*\x+1},-2) -- ({2*\x+3},2);
            \draw ({2*\x+1},2) -- ({2*\x+3},-2);
            \draw [fill=white] ({2*\x+2},0) circle (2pt);
            \draw [very thick,<-] ({2*\x+1.9},-2) -- ({2*\x+2.1},-2);
        }
        \draw [very thick,<-] (9.9,2) -- (10.1,2);
        \draw [very thick,<-] (-0.05,0.1) -- (0.05,-0.1);
        \draw [very thick,<-] (11.45,1.1) -- (11.55,0.9);
        \draw [very thick,<-] (11.45,-1.1) -- (11.55,-0.9);
        \end{scope}
    \end{tikzpicture}
    \caption{Adding a zigzag (black) between two intersecting precells (red and blue).}
    \label{fig:zigzag}
\end{figure}

\begin{proof}[Proof of Corollary \ref{cor:application}]
    By Theorem~\ref{thm:s.e.s.}, to prove Corollary \ref{cor:application} it is enough to show that the action subgroup $A_G = \{\varphi_g \mid g \in G\} \leqslant \Aut(A_\Gamma)$ is orbit decidable.
    Let $O_G \leqslant \Aut(A_\Gamma)$ be a set of unique representatives of the projection of $A_G$ in $\Out(A_\Gamma)$.
    Then given $u,v\in A_\Gamma$, checking whether there is an automorphism in $A_G$ that sends $u$ to $v$ is equivalent to checking whether $v$ is conjugate to some element in $\{\varphi(u) \mid \varphi \in O_G\}$.
    Since $\Out(A_\Gamma)$ is finite (by Theorem~\ref{thm:Out-finite}), this problem is algorithmically solvable.
\end{proof}

\bibliographystyle{amsalpha}
\bibliography{mybib}

\end{document}